\documentclass{amsart}[12pt,a4paper]

\usepackage[applemac]{inputenc}
\usepackage[colorlinks, linktocpage, citecolor = blue, linkcolor = blue]{hyperref}

\usepackage[english]{babel} 
\usepackage{amsfonts}  
\usepackage{amsmath}
\usepackage{amssymb} 
\usepackage{mathrsfs} 
\usepackage{color}  
\input xy
\xyoption{all}
\usepackage{amsthm}
\usepackage{latexsym}
\usepackage[T1]{fontenc} 
\usepackage{graphicx}
\usepackage{graphics} 
\usepackage{fancyhdr}
\input xy
\xyoption{all}
\usepackage{enumerate}

\newcommand{\Q}{\mathbb{Q}}
\newcommand{\Z}{\mathbb{Z}}

\newcommand{\C}{\mathbb{C}}

\newcommand{\De}{\Delta}

\DeclareMathOperator{\Pic}{Pic}
\DeclareMathOperator{\NS}{NS}

\DeclareMathOperator{\Mon}{Mon}

\DeclareMathOperator{\Aut}{Aut}
\DeclareMathOperator{\GL}{GL}

\DeclareMathOperator{\NE}{\overline{NE}}
\DeclareMathOperator{\Hom}{Hom}

\DeclareMathOperator{\Galois}{Gal}

\newtheorem{theorem}{Theorem}[section]
\newtheorem{lemma}[theorem]{Lemma}
\newtheorem{proposition}[theorem]{Proposition}
\newtheorem{question}[theorem]{Question}

\newtheorem{corollary}[theorem]{Corollary}

\newtheorem{rem}[theorem]{Remark}

\theoremstyle{definition}

\newtheorem{definition}[theorem]{Definition}

\theoremstyle{remark}
\newtheorem{remark}[theorem]{Remark}

\newcommand{\qq}{\mathbb{Q}}
\newcommand{\zz}{\mathbb{Z}}

\newcommand{\pp}{\mathbb{P}}

\newcommand{\oo}{\mathcal{O}}

\newcommand{\ch}{\operatorname{char}}

\begin{document}

\title{A note on the fibres of Mori fibre spaces}

\author[G. Codogni]{Giulio Codogni}
\address{EPFL, SB MATHGEOM CAG, MA B3 635 (B\^{a}timent MA), Station 8, CH-1015 Lausanne, Switzerland.}
\email{giulio.codogni@epfl.ch}

\author[A. Fanelli]{Andrea Fanelli}
\address{Mathematisches Institut, Heinrich-Heine-Universit\"at D\"usseldorf,
Universit\"atsstr.~1, 
40204 D\"usseldorf, Germany.}
\email{fanelli@hhu.de}

\author[R. Svaldi]{Roberto Svaldi}
\address{DPMMS - University of Cambridge 
Centre for Mathematical Sciences, 
Wilberforce Road, 
Cambridge CB3 0WB, UK.}
\email{rs872@cam.ac.uk}

\author[L. Tasin]{Luca Tasin}
\address{Mathematical Institute of the University of Bonn, Endenicher Allee 60, D-53115
Bonn, Germany.} 
\email{tasin@math.uni-bonn.de}


\begin{abstract}
In this note we consider the problem of determining which Fano manifolds can be realised as fibres of a Mori fibre space. 
In particular, we study the case of toric varieties, Fano manifolds with high index and some Fano manifolds with high Picard rank.  
\end{abstract}

\maketitle

\tableofcontents


\section*{Introduction}
In algebraic geometry, one of the main goals is to classify algebraic varieties. 
Rather than distinguishing varieties based on their isomorphism type, one
can look at their birational type,
i.e. the structure of a non-empty Zariski open set of the variety, which allows more
flexibility in the choice of preferred model for the object under scrutiny.\newline
From this point of view, the study of the canonical bundle of a smooth projective 
variety, that is, the determinant of the cotangent bundle, plays a central role in 
the classification. There is a stark dichotomy between algebraic varieties that admit
sections of powers of the canonical bundle and those that do not. \newline
In fact, while it is expected that the former are birational to a fibration in Calabi-Yau
varieties over a base of (log-)general type, using a suitable realization of the Iitaka fibration, 
the latter are instead expected to be birational to a fibration in Fano varieties
over a smaller dimensional base. \newline 
The following definition gives the precise notion needed in the latter case.
Let us remind the reader that a normal projective variety is said to be Fano
if the anticanonical divisor is ample (and in particular it is $\mathbb{Q}$-Cartier).
\begin{definition}
Let $f\colon X \to Y$ be a dominant projective morphism of normal varieties. 
Then $X$ is a \emph{Mori fibre space} (or \emph{MFS}) if
\begin{itemize}
\item $f_*\oo_X = \oo_Y$ and $\dim Y < \dim X$;
\item $X$ is $\qq$-factorial with klt singularities;
\item the general fibre is a Fano variety and $\rho(X/Y):=\rho(X)-\rho(Y)=1$.
\end{itemize}
\end{definition}
While it is a difficult problem to prove that a variety having no non-zero pluricanonical 
forms is birational to a Mori fibre space, 
in \cite{MR2601039} the authors show, among other things, that this is actually the case if instead
the variety is assumed to have a non-pseudoeffective canonical bundle -- the two conditions
being expected to be equivalent as predicted by the existence of minimal models and
the celebrated Abundance Conjecture, cf. \cite[Conjecture 3.12]{kollarmori}.
In view of these considerations, it is natural to wonder about the following matter.
\begin{question}\label{intro.quest}
What type of Fano varieties appear as fibers in a Mori fibre space?
\end{question}
Part of the difficulty in answering the above question lays in the lack of clarity 
as to what fibers of a MFS one should actually consider.
To this end, the notion of fibre-like Fano variety was introduced in \cite{CFST_16}. 
For the sake of simplicity, we only consider the case of smooth fibers.

\begin{definition}\label{fibre-like}
A Fano manifold $F$ is \emph{fibre-like} if it can be realised as a fibre of a Mori fibre space $f\colon X \to Y$ over the smooth locus of $f$.
\end{definition}

Any Fano manifold with Picard number $\rho=1$ is fibre-like, via the constant map to a point. 
When $\rho \ge 2$ the problem of determining whether a variety is fibre-like or not is highly non-trivial. 
Mori showed in \cite[Theorem 3.5]{Mori_surf} which Fano surfaces are fibre-like. In \cite{CFST_16}, 
we systematically study fibre-like Fano varieties 
by analyzing the action of the monodromy of the MFS 
on the N\'eron-Severi group of a general fibre. 
Moreover, we fully characterize those threefolds that are fibre like, cf. \cite[Theorem 1.4]{CFST_16} 
and give the following sufficient condition for fibre-likeness in any dimension.

\begin{theorem}[\cite{CFST_16}, Theorem 3.1]\label{suff_criterion}
A smooth Fano variety $F$ is fibre-like if
\begin{equation}\label{eq.intro}
\NS(F)_\qq^{\Aut(F)}=\qq K_{F}. 
\end{equation}
When $F$ is rigid, then property \eqref{eq.intro}
is equivalent to $F$ being fibre-like.
\end{theorem}

In \cite{CFST_16}, we also establish that the fibre-likeness of a Fano manifold $F$ implies an analogous necessary condition: namely, 
$F$ fibre-like implies 
$$\NS(F)_\qq^{\Mon(F)}=\qq K_{F},$$ 
where $\Mon(F)$ is the maximal subgroup of $\GL(\NS(F),\zz)$ which preserves the birational data of $F$.\newline 
Below we recall some of the consequences of our analysis. First, we construct a large class of examples of fibre-like Fano manifolds.

\begin{corollary}[\cite{CFST_16}, Corollary 4.6]\label{complete_intersections}
Let $r,k,d$ be integers with $n\ge 2$ and $kd < n+1$. Then any smooth complete intersection of $k$ divisors of degree $(d, \ldots , d)$ in $(\pp^n)^r$ is fibre-like.
\end{corollary}

Moreover, since the notion of fibre-likeness is strictly intertwined with the monodromy action, we show that 
it forces a high degree of symmetry on facets of the nef cone of the variety.
Recall that a facet of a polyhedral cone is just a maximal dimensional face.
As Fano varieties are Mori dream spaces, any facet $\mathcal G$ of the nef cone $F$ corresponds to a contraction 
$\pi: F \to G$ such that $\mathcal G = \pi^* \mathrm{Nef}(G)$.

\begin{corollary}[\cite{CFST_16}, Corollary 3.9] \label{contraction}
Let $F$ be a smooth fibre-like Fano variety. 
Let $\mathcal G$ be a facet of the nef cone of $F$ corresponding to a contraction $F \to G$. 
Then for any other facet $\mathcal H$ with contraction $F \to H$ we have that $G$ and $H$ are deformation equivalent.
\end{corollary}
Let us point out that the definition of fibre-like can be extended to singular varieties (cf. \cite{CFST_16}, Definition 2.14) 
and all results illustrated so far still hold in that setting.

\smallskip
In this note we show that various natural and interesting classes of Fano manifolds are fibre-like.
\smallskip

We first focus on smooth toric Fano varieties. If $F(\De)$ is a smooth toric Fano variety with associated polytope $\De$, 
then $F$ is said to be \emph{vertex-transitive} if the automorphism group of the polytope acts transitively on the set of vertices of $\De$. 
A vertex-transitive toric Fano manifold is fibre-like (see Lemma \ref{lem:vt}) and all the known examples of toric fibre-like Fano varieties 
seem to be of this kind (cf. Question \ref{v_tran}).

Among these, projective spaces and $t$-del Pezzo manifolds (see Definition \ref{def:del Pezzo}) 
are classical examples of vertex-transitive varieties.
A generalisation of $t$-del Pezzo manifolds has been introduced in \cite{Kl83} and studied in \cite{VK85}. 
We call those Klyachko varities (see Definition \ref{skansen_var}). We show that Klyachko varities are vertex-transitive 
and we show that they constitute a fundamental building block in the theory of fibre-like toric Fano manifolds. 

\begin{theorem}[Propositions \ref{2-neighbourly} and \ref{classification_low}]
Let $F=F(\De)$ be a $d$-dimensional vertex-transitive Fano manifold. 
\begin{enumerate}[(i)]
\item If there are two vertices of $\De$ that are not in the same face (i.e.\ $\De$ is not 2-neighbourly), then $F$ is a power of $\mathbb P^1$ or a power of t-del Pezzo manifolds. 
\item If $d \le 7$, then $F$ is a power of projective spaces or Klyachko varieties.
\end{enumerate}
\end{theorem}

\smallskip

Next we look at Fano manifolds with high index. 
As mentioned above, fibre-likeness is completely understood for surfaces and threefolds.
Hence, here we focus on Fano varieties of dimension at least~4.

For a Fano manifold $F$, the index $i_F$ is defined as the largest integer that divides $-K_F$ in $\mathrm{Pic}(F)$.
The index is one of the most basic numerical invariants for Fano varieties:
there is a complete classification of Fano manifolds with Picard number at least 2 and index $i \ge n-2$ due to 
Kobayashi-Ochai \cite{KO_73} , Fujita \cite{F_82a,F_82b} and Mukai \cite{M_89}. 
Moreover, Wi\'sniewski \cite{Wis_91} classified Fano manifolds with index $i \ge (n+1)/2$. 
Notice that the only case for which  $i < (n+1)/2$ and $i \ge n-2$ is $i=2$ and $n=4$.
 
When the Fano index is greater then the above bounds, we show that
fibre-like Fano varieties can be explicitly classified.
 
\begin{theorem}[Propositions \ref{wis index} and \ref{mukai}]
Let $F$ be a fibre-like Fano manifold with Picard number $\rho \ge 2$ and dimension $n \ge 4$.
\begin{itemize}
\item If $i_F \ge \frac{n+1}{2}$, then $F \cong \mathbb P^{n/2} \times \mathbb P^{n/2}$ ($n$ even) or $F\cong \mathbb P \Big(T_{\mathbb P^{\frac{n+1}{2}}}\Big)$ ($n$ odd).

\item If $n=4$ and $i_F=2$ (i.e.\ $F$ is a Mukai fourfold), then $F$ is isomorphic to one of the following:
\begin{enumerate}[(i)]
\item a double cover of $\pp^2 \times \pp^2$ branched along a degree--$(2,2)$ divisor;
\item an intersection of two degree--$(1,1)$ divisors in $\pp^3\times\pp^3$;
\item $\pp^1\times\pp^1\times\pp^1\times\pp^1$.
\end{enumerate}

\end{itemize} 
\end{theorem}

\smallskip

Finally, we discuss families of Fano manifolds with high Picard number. \newline
As there only exist finitely many families of Fano manifolds in any given dimension, 
there must be an upper bound on the Picard number in fixed dimension. 
It is a natural question to ask whether it is possible to compute such bound. \newline
The only known examples of Fano manifolds -- other than products -- of dimension 
$n\geq 4$ and Picard number at least $n+5$ are birational modifications of the 
blow-up of $\mathbb{P}^4$ at $8$ points in general position. 
These manifolds have been studied in \cite{CCF17}; among other results, the authors showed that they are fibre-like. 
Here, we study the only known examples of Fano manifolds -- again, other than products --
that have dimension $n\geq 4$ and Picard number at least $n+4$. 
These are birational modification of $\mathbb{P}^n$ blown up at $n+3$ general points, with $n$ even. 
They are isomorphic to the space of $(m-1)$-planes 
in the intersection of two quadrics  in $\pp^{n+2}$, where $n=2m$. 
Building on \cite{AC17}, we can prove the following result.
\begin{theorem}[Theorem \ref{main_AC}]
Let $n\geq 4$ be an even integer. An $n$ dimensional Fano manifold birational to the blow-up of $\mathbb{P}^n$ at $n+3$ general points is fibre-like.
\end{theorem}

\smallskip
The structure of the paper is as follows. Section~\ref{sec_toric} is dedicated to toric Fano manifolds.
In Section~\ref{sec_index} we discuss the fibre-likeness of Fano manifolds with high index.
Section~\ref{sec_fourfolds} is devoted to the study of certain families of Fano manifolds with high Picard number.
Finally, in Section~\ref{sec_questions} we present new questions and research directions.

\subsection*{Acknowledgement}

We like  to thank Cinzia Casagrande for interesting comments on this work. 
The second-named author has been supported by the DFG grant ``Gromov-Witten Theorie, Geometrie und Darstellungen'' (PE 2165/1-2) and by the Swiss National Science Foundation Grant ``Algebraic subgroups of the Cremona groups'' (200021--159921).
The fourth-named author is supported by the DFG grant ``Birationale Methoden in Topologie und Hyperk\"ahler Geometrie''.

\subsection*{Notation} 

The term \emph{variety} stays for separated, integral and proper scheme of finite type over $\C$. 
A \emph{manifold} is a smooth variety.
For the definitions of singularities in the context of the minimal model program, see \cite[Section 2]{K13}.


\section{Toric Fano varieties}
\label{sec_toric}

The first class of examples we want to discuss is given by toric varieties. 
We will recall here only some notions and refer to \cite{CLS11} for an exhaustive treatment of the topic. 
Toric varieties can be described in terms of combinatorial data and this makes them particularly suitable to test general conjectures on Fano varieties.

Let $N$ be a free abelian group of rank $n$ and set $N_\qq:=N\otimes_\zz \qq$. 
Let $\Sigma \subset N_\qq$ be a fan of a $d$-dimensional toric Fano variety $F$ and let $\De$ 
be the polytope associated to the anti-canonical polarisation. 
Furthermore, $M$ will denote the dual of $N$. 

The vertices of $\De$, denoted by $V(\De)$, are the generators of $\Sigma$. 
We denote by $O(\sigma)$ the closure of the orbit corresponding to $\sigma\in \Sigma$, which is an irreducible invariant subvariety.

Let $A_1$ be the group of 1-cycles on $F$ modulo numerical equivalence and set $N_1=A _1\otimes \qq$. 
Inside $N_1$, we consider the Kleiman-Mori cone $\NE(F)$ generated  by the effective 1-cycles. 
We have the following standard exact sequence (cf. \cite[Chapter~4, Theorem 1.3]{CLS11})
\begin{equation}\label{toricses1}
0 \rightarrow A_1(F) \rightarrow \zz^{V(\De)}\rightarrow N \rightarrow 0,
\end{equation}
which by duality yields the following one:
\begin{equation}\label{toricses2}
0 \rightarrow M \rightarrow \zz^{V(\De)}\rightarrow \NS(F) \rightarrow 0.
\end{equation}

\subsection{Vertex-transitive Polytopes}

The following definition of vertex-transitive for polytopes is classical, 
although some authors refer to them as isogonal polytopes (cf. \cite[19.5, Enumeration]{GO_04}).

\begin{definition}
A polytope $\Delta$ is \emph{vertex-transitive} if $\Aut(\Delta)$ acts transitively on the vertices of $\Delta$. If $\Delta$ is associated to a toric Fano
variety $F$, then $F$ is vertex-transitive.
\end{definition}

This class of varieties is interesting from our prospective for the following reason. 

\begin{lemma}\label{lem:vt}
Vertex-transitive Fano manifolds are fibre-like.
\end{lemma}
\proof
Let $F(\De)$ be a Fano toric variety and let  $G=\Aut(\De)$. 
As explained in the proof of \cite[Theorem 5.7]{CFST_16}, the exact sequence in \eqref{toricses2} yields the following sequence
\[
0 \rightarrow M_\qq^G \rightarrow (\qq^{V(\De)})^G \rightarrow \NS(F)_\qq^G \rightarrow 0.
\]
Then, $F$ is fibre-like if and only if $t-k=1$, where  $t$ is the number of orbits of the action of $G$ on $V(\De)$ and $k=\dim M_\qq ^G$.

If $F$ is vertex-transitive, then $t=1$ and so \cite[Lemma 5.10]{CFST_16} implies that $k=0$, which means that $F$ is fibre-like. 
\endproof

Denote by $d$ the dimension of the toric Fano variety $F$ and by $m$ the number of vertices of $\Delta$. 

The first non-trivial class of vertex-transitive Fano varieties are {\it $t$-del Pezzo manifolds}.
\begin{definition}\label{def:del Pezzo}
The $d$-dimensional \emph{$t$-del Pezzo manifold} $V_d$ (with $d$ even) is the smooth toric Fano variety whose associated polytope has vertices
\[
V(\De)=\{e_1, \ldots, e_d, -e_1, \ldots, -e_d, (e_1+\ldots +e_d), -(e_1+\ldots +e_d) \},
\]
 where $e_1, \ldots, e_d$ is the standard basis of $N_\qq$. 
\end{definition}

\begin{remark}
In the literature on toric geometry, these manifolds are simply named {\it del~Pezzo varieties}. 
We added the prefix ``$t-$'' in order to distinguish them from the \emph{del~Pezzo manifolds} appearing in Section~\ref{sec_index}.
\end{remark}

$T$-del Pezzo polytopes are symmetric with respect to the origin, i.e., $-\De = \De$. Polytopes satisfying this condition are 
also said to be \emph{centrally symmetric}, cf. \cite{VK85}.
A classic result by Voskresenkii and Klyachko shows that del Pezzo varieties are essentially the only centrally symmetric toric varieties.

\begin{theorem}[{\cite[Theorem 6]{VK85}}]\label{veve}
Let $F$ be a toric Fano manifold such that $\De$ is centrally symmetric. Then $F$ is isomorphic to a product of projective lines and $t$-del Pezzo varieties.
\end{theorem}

Coming back to vertex-transitive varieties, we can prove the following structural result.

\begin{lemma}\label{product}
Let $F=F(\De)$ be a toric Fano manifold which is vertex-transitive. Then there exists a unique vertex-transitive Fano toric manifold $F_{min}$ and a positive integer $n$ such that $F \cong (F_{min})^{n}$.  
\end{lemma}
\begin{proof}
Let $\De= \De_1^{n_1} \times \ldots \times \De_r^{n_r}$ be a prime decomposition of $\De$. The automorphism group of $\De$ is given by (see for example \cite[Theorem A]{GH16})
$$
\Aut(\De)= \prod_{i=1}^r (\Aut(\De_i) \rtimes S_{n_i})
$$
with its natural action on $\De$. 

Since $\Aut(\De)$ acts transitively on the vertices of $\De$, the lemma follows immediately.
\end{proof}

\subsection{Primitive collections}

Primitive collections are an essential tool to study the birational geometry of Fano toric varieties. We refer to \cite{Reid83} and \cite{Cas03a} for further details.

\begin{definition}
Let $F=F(\De)$ be a toric Fano variety. A subset $P \subset V(\Delta)$ is called a \emph{primitive collection} if the cone generated by $P$ is not in $\Sigma$, but for any $x \in P$ the elements of $P \setminus \{x\}$ generate a cone in $\Sigma$. 

For a primitive collection $P=\{x_1, \ldots,x_k\}$ denote by $\sigma(P)$ the (unique) minimal cone in $\Sigma$ such that $(x_1+\ldots +x_k) \in \sigma(P)$ . Let $y_1,\ldots,y_h$ be generators of $\sigma(P)$, then
\begin{equation}\label{linear relation}
r(P)\colon x_1+\ldots +x_k=b_1y_1+\ldots +b_hy_h
\end{equation}
where $b_i$ is a positive integer for all $1 \le i \le h$: we have simply written the element $x_1+\ldots +x_k$ in terms of the generators $y_1,\ldots,y_h$ (the coefficients are positive since $(x_1+\ldots +x_k)$ is in the cone $\sigma(P)$).

The linear relation \eqref{linear relation} is called \emph{the primitive relation of $P$} and the cone $\sigma(P)$ is called the focus of $P$. 
The integer $k$ is called the {\it length of $r(P)$} and the {\it degree of $P$} is defined as $\deg P= k - \sum b_i$.
\end{definition}

Here it is convenient to write down explicitly the group of $1$-cylces $A_1$ of $F$ as:
$$
A_1(F) \cong \left\lbrace (b_x)_{x \in V(\De)} \in \Hom(\zz^{m}, \zz)\   \Bigg| \ \sum_{x \in V(\De)} b_x x=0\right\rbrace .
$$
The previous isomorphism is clear looking at the exact sequence \eqref{toricses1}. So it is natural to identify primitive relations with the associated cycles. Moreover, we work on Fano varieties, so $\deg P = -(K_F \cdot r(P)) >0$ for all primitive relations.

Consider now a primitive collection $P$ on $F$ for which the relation $r(P)$ is {\it extremal}, meaning that it generates an extremal ray in $\NE(F)$. One sees that the exceptional locus of the associated contraction is given by $O(\sigma(P))$ and moreover, according to the dimension of $\sigma(P)$, one recovers:
\begin{itemize}
\item \emph{divisorial contraction} when $\sigma(P)$ is a one-dimensional cone and the contracted divisor is precisely the one associated to the ray;
\item \emph{Mori fibration}, when $\sigma(P)$ coincides with the origin;
\item \emph{flipping contraction} otherwise.
\end{itemize} 

Let us recall some useful results.

\begin{proposition}[{\cite[Prop. 4.3]{Cas03a}}]\label{deg1}
Let $\gamma\in \NE(F)\cap A_1(F)$ be a $1$-cycle of $F$ for which $(K_F\cdot \gamma)=-1.$ Then $\gamma$ is extremal.
\end{proposition}

\begin{theorem}[{\cite[Theorem 2.4]{Reid83}}, {\cite[Theorem 1.5]{Cas03a}}]\label{sum}
Let $R\subset\NE(F) $ be an extremal ray and let $\gamma \in R \cap A_1(F)$ be a primitive cycle. Then there exists a primitive collection $P=\{x_1, \ldots, x_k\}$ such that 
$$
\gamma=r(P)\colon  x_1+\ldots +x_k=b_1y_1+\ldots b_hy_h.
$$ 
Moreover, for any cone $\nu=\langle z_1, \ldots,z_t \rangle$ which verifies
\begin{itemize}
\item  $\{z_1,\ldots,z_t\} \cap \{x_1,\ldots,x_k, y_1, \ldots, y_h\}= \emptyset$; and 
\item $\langle y_1,\ldots,y_h\rangle + \nu \in \Sigma$; 
\end{itemize}
the following holds for all $i=1, \ldots, h$:
$$
\langle x_1, \ldots, \check{x_i}, \ldots, x_k,y_1\ldots , y_h\rangle + \nu \in \Sigma.
$$
\end{theorem}

\begin{proposition}[{\cite[Prop. 3.4]{Cas03a}}]\label{difference}
Let $P$ a primitive extremal collection for $F$ and write $\sigma(P)=\langle y_1,\ldots,y_h\rangle$. Then for any other primitive collection $Q\neq P$ for which $P \cap Q \ne \emptyset$, the set $(Q \setminus P) \cup  \{y_1,\ldots,y_h\}$ contains a primitive collection.
\end{proposition}

The following observation is easy but useful for our analysis.

\begin{rem}\label{rmk}
Consider a relation
$$a_1x_1+\cdots+a_kx_k=b_1y_1+\cdots+b_hy_h$$
among the vertices of $\De$, with $a_i,b_j>0$ for all $i,j$. 

If $\sum a_i \ge \sum b_j,$
then \cite[Lemma 1.4]{Cas03a} implies that $\langle x_1\cdots, x_k\rangle \not\in \Sigma$.
\end{rem}

We recall now the following definition.

\begin{definition}[$k$-neighbourly polytope]
A polytope is \emph{$k$-neighbourly} if every set of $k$ vertices lies on one of its face. \newline
A Fano variety $F(\Delta)$ is $k$-neighbourly if the corresponding polytope $\Delta$ is.
\end{definition}

We want now to understand the structure of vertex-transitive polytopes: the following result is the first step towards a classification of vertex-transitive those.

\begin{proposition}\label{2-neighbourly} 
Let $F=F(\Delta)$ be a vertex-transitive toric Fano manifold. Then either
\begin{enumerate}
\item $F=(\pp^1)^d$ or $F=(V_k)^r$ for some $r$ and $k$ or
\item $\De$ is 2-neighbourly. 
\end{enumerate}
\end{proposition}

\begin{proof}
Let us assume that $\De$ is not 2-neighbourly, which implies the existence of a primitive collection with two elements. We claim that this primitive relation can be assumed to be of the form 
\begin{equation}\label{2-neigh}
x+y=0.
\end{equation}
To show this, we assume such a relation does not exist and seek for contradiction. Take a primitive collection $P_1=\{x_1,x_2\}$ verifying the relation $R_1\colon x_1+x_2=y_1$. Let $\Aut(\De)$ act on $P_1$ to obtain a family of primitive collections $\mathcal P= \{ P_i \}_{1\le i \le r }$ with relations $\mathcal R= \{ R_i \}_{1\le i \le r }$. 
Since the action is transitive by hypothesis, any vertex of $\De$ appears the same number of times as right hand side of these relations and so the number of vertices $m:=\#\{V(\De)\}$ divides $r$. This implies that the $P_i$'s cannot be all disjoint, otherwise $2r=m$. Hence we may assume that $P_2=\{x_1,x_3\}$, $x_2 \ne x_3$, with relation $R_2\colon x_1+x_3=y_2$. 
\\The two relations $R_1$ and $R_2$ give $x_2+y_2=x_3+y_1$, which implies, by Remark \ref{rmk}, that $\{x_2,y_2\}$ is also a primitive collection. The two relations
$$R'\colon x_3+y_1=z_1 \mbox{ \ \ and  \ \ } R''\colon x_2+y_2=z_1$$
are extremal, so Proposition \ref{deg1} and Theorem \ref{sum} imply that $\langle y_1,y_2\rangle \in \Sigma$.
\\This is a contradiction, since  $y_1+y_2=x_1+z_1.$ We proved the existence of \eqref{2-neigh}. 

Now act with $\Aut(\Delta)$ to get exactly $m/2$ relations of the same form. One can verify that those are disjoint. Using the vertex-transitivity of $\De$, we deduce that for any vertex $\bar x$ there is a vertex $\bar y$ for which $\bar x+\bar y=0$, i.e. $\Delta$ is centrally symmetric. Theorem \ref{veve} concludes the proof.

\end{proof}

We study now the extremal contractions of 2-neighbourly vertex-transitive toric Fano manifolds.

\begin{lemma}\label{extremal_lemma}
Let $F=F(\Delta)$ be a vertex-transitive, 2-neighbourly toric Fano manifold. Then there exist an integer $k \ge 3$ and a set of primitive collections 
$$\mathcal P =\{P_i\}_{i=1,\ldots,r}$$
such that $r=m/k$, $|P_i|=k$, $\sigma(P_i)=0$ and $P_i \cap P_j= \emptyset$ for any $i \ne j$. Moreover, these are the only primitive relations with focus equal to zero. 
\end{lemma}

\begin{proof}
The result in \cite[Proposition 3.2]{Bat91} implies that there exists a primitive collection $P_1$ with $\sigma(P_1)=0$. Define $k:=|P_1|$. Since $\De$ is 2-neighbourly, we have $k \ge 3$.

Act with $\Aut(\De)$ to get a set of primitive collections $\mathcal P= \{ P_i \}_{1\le i \le r }$ verifying $\sigma(P_i)=0$ and for which $\bigcup_{i=1}^r P_i=V(\De)$. Let us prove they are disjoint, assuming that $P_i \cap P_j \ne \emptyset$ for some $i,j$ and seeking for contradiction. Write $P_i=\{x_1, \ldots, x_k\}$ and $P_j=\{x_1, \ldots, x_h, y_{h+1}, \ldots ,y_k\}$ with $y_s \ne x_t$ for any $s,t$. Then 
$$
x_{h+1}+\ldots +{x_k}=y_{h+1}+\ldots +y_k.
$$ 
Remark \ref{rmk} gives the required contradiction, since $x_{h+1}, \ldots, x_k$ generate a cone in $\Sigma$. Moreover one sees that there are no other primitive relations with focus equal to zero. 
\end{proof}

\begin{proposition}\label{extremal}
In the notation of Lemma \ref{extremal_lemma}, assume that one of the relations $P_i$ is extremal. Then $F(\De)=(\pp^{k-1})^r$. On the other hand, if any of these relations is not extremal then $F$ does not admit any extremal contraction of fibre type. 
\end{proposition}

\begin{proof}
Up to reordering, assume that $P_1$ is extremal. Acting with $\Aut(\De)$, we deduce that all $P_i$'s are extremal. We claim these are the only primitive collections. In fact, let $\tilde P$ be a primitive collection such that $\tilde P \notin \mathcal P$ and $\tilde P$ has minimal cardinality among the primitive collections which are not in $\mathcal P$.  We may assume that $P_1 \cap \tilde P \ne \emptyset$. Using Proposition \ref{difference}, we deduce that the set $(\tilde P \setminus P_1)$ contains a primitive collection. Contradiction, since $|\tilde P|$ is minimal.   
\\Note that $k$ is the index of $K_F$, $\dim F=d=(k-1)r$ and $\rho(F)=r$ by \cite[Corollary 4.4]{Bat91}. So apply \cite[Theorem 1]{Cas06} (Mukai's conjecture) to obtain the first part of the statement.

For the last part, just observe that an extremal contraction of fibre type would provide a primitive collection $P$ with trivial focus $\sigma(P)=0$.
\end{proof}

\begin{lemma}\label{divisorial}
Let $F=F(\Delta)$ be a vertex-transitive, 2-neighbourly toric Fano manifold. Then there are no extremal relations of the form
\begin{equation}\label{div_2_neigh}
x_1+\ldots+x_k=by_1.
\end{equation}
In particular $F$ does not admit any extremal divisorial contraction.
\end{lemma}

\begin{proof}
Let us assume that an extremal relation of the form \eqref{div_2_neigh} exists and seek for contradiction. 
Let  $\mathcal R= \{ R_i \}_{1\le i \le r }$ be the set of extremal relations obtained acting with $\Aut(\De)$ and denote with $P_i$ the associated collections. 
Assume that $x_1$ appears only in one $P_i$. Since by transitivity any vertex appears the same number of times, we get that the $P_i$'s are disjoint. In particular $m=kr$, where $m=\#\{V(\De)\}$.
On the other hand, we have that $m$ divides $r$, because any vertex appears the same number of times as right hand side.
This implies $k=1$, which is a contradiction.

Hence there is an extremal primitive relation different from \eqref{div_2_neigh} of the form
$$
x_1+z_2\ldots+z_k=by_2.
$$

Assume $y_2 \notin \{x_2, \ldots,x_k,y_1\}$ (the other case is analogous). \\We have $ b_1y_2+x_2+\ldots +x_k=by_1 + z_2+\ldots+z_k$, and, since $\De$ is 2-neighbourly, we know that $\langle y_1,y_2 \rangle$ is a cone of $\Sigma$. Theorem \ref{sum} implies that $\langle y_2,x_2, \ldots,x_k\rangle \in \Sigma$, but this contradicts Remark \ref{rmk}.

\end{proof}

\subsection{Klyachko varieties}

Looking for interesting examples of vertex-transitive toric varieties, we found a generalisation of $t$-del Pezzo varieties, which were introduced in \cite{Kl83} and studied in \cite{VK85}.
\\Let us remark that our notation is not the same as Klyachko (cf. Remark \ref{confronto}). Fix a basis $e_1, \ldots, e_d$ of a lattice $N \cong \Z^d$, with $d\ge 2$ and let $k$ be a positive integers such that $(k-1)|d$.

\begin{definition}\label{skansen_var}
The \emph{Klyachko variety of order $k$ and dimension $d$} is the toric Fano variety $W^k_d$ with polytope $\Delta^k_d \subset N$ having vertices
\begin{align*}
V(\De^k_d)= \{& e_1,e_2, \ldots, e_d, e_1+\ldots +e_d, \\
              & -(e_1+\ldots+e_{k-1}), -(e_k+\ldots+e_{2k-2}), \ldots, -(e_{d-k+2} + \ldots +e_d), \\
              &  -(e_1 + e_k + \ldots + e_{d-k+2}), -(e_2 + e_{k+1} + \ldots + e_{d-k+3}), \ldots, \\ 
              &-(e_{k-1}+e_{2k-2}+\ldots + e_{d}) \}.
\end{align*}  
\end{definition}

\begin{rem}\label{confronto}
When $d$ is even, $W_d^2$ is the $t$-del Pezzo manifold $V_d$.
\\In \cite{VK85}, the varieties $W^k_d$ are introduced as $P_{m,n}$. The dictionary between the indices is:
$$d=(m-1)(n-1), \ \ \ k=m \ \mbox{(or \ $n$)}.$$
As we will see in Lemma \ref{skanseniso}, our definition of $k$ is consistent.
\end{rem}

If $W^k_d$ is smooth (cf. Proposition \ref{smoothskansen}), we can describe some birational geometry of Klyachko varieties.
\\The 1-dimensional cones of the fan of $W^k_d$ coincide with the 1-dimensional cones of the fan of the blow-up $Z^k_d$ of $(\mathbb{P}^{k-1})^{\frac{d}{k-1}}$ in $k$ invariant points. This implies that $W^k_d$ and $Z^k_d$ are isomorphic in codimension one and $W^k_d$ is a Fano model of $Z^k_d$ (cf. Section \ref{sec_fourfolds} for other examples of fibre-like Fano manifolds obtained as small modifications of blow-ups of projective spaces).

\begin{lemma}\label{skanseniso}
For any integers $d$ and $m$, $W^{d+1}_{md} = W^{m+1}_{md}$.
\end{lemma}

\proof
Assume $m\le d$ and consider the vertices of $W^{d+1}_{md}$:
\begin{align*}
\{& e_1,e_2, \ldots, e_{md}, \ e_1+\ldots +e_{md}, \\
              & -(e_1+\ldots+e_{d}), -(e_{d+1}+\ldots+e_{2d}), \ldots, -(e_{m(d-1)+1} + \ldots +e_{md}), \\
              &  -(e_1 + e_{d+1} + \ldots + e_{m(d-1)+1}), -(e_2 + e_{d+2} + \ldots + e_{m(d-1)+2}), \ldots, \\ 
              &-(e_{d}+e_{2d}+\ldots + e_{md}) \}.
\end{align*}  
The following transformation
$$e'_{mi+j} := e_{d(j-1)+i+1}$$
where $i \in \{0,\ldots,d\}$ and $j\in \{1,\ldots,m-1\}$ gives the identification.

\endproof

We study now symmetries and singularities of Klyachko varieties.

\begin{lemma}
The Fano varieties $W_d^k$ are vertex-transitive, reflexive and have terminal singularities, for all $d,k$. 
\end{lemma}
\begin{proof}
Let us fix $k$ and observe that Lemma \ref{skanseniso} provides the following identification: $W_{k-1}^k \cong W_{k-1}^2$. 

Vertex-transitivity is proved by induction on $d$: assume that for any $(k-1)|d'$ and $d'<d$, the variety $W_{d'}^k$ is vertex-transitive. We write the projections 
$$\pi_i\colon \Delta_d^k \longrightarrow \langle e_{i+1},e_{i+2},\ldots,e_{i+k-1} \rangle^{\perp},$$ 
with $i=0,\ldots,d-k+1$.
By inductive hypothesis, the images via the $\pi_i$'s of $\De_d^k$ are vertex-transitive and they are all isomorphic to $\De_{d-k+1}^k$. To prove the transitivity for the whole polytope, we act with $\GL(N_\qq)$ to exchange the subspaces $\langle e_{i+1},e_{i+2},\ldots,e_{i+k-1} \rangle.$

\medskip
We write now $W=W_d^k$ and $\De=\De^k_d$ to simplify the notation and prove reflexivity. Look at the dual polytope $\De^* \subset M_\qq$: we claim that no lattice point lies between the affine hyperplane spanned by
the facets of $\De^*$ and its parallel through the origin. The claim holds for the hyperplane $\{x_1=-1\} \subset M_\qq$, so acting with $\Aut(\De)$ on $\De^*$ we conclude.

Terminality can be translated on polytopes with the condition
$$\De \cap N = V(\De) \cup \{0\}.$$ 
We assume there exists a non-zero $v \in \De \cap N$ which verifies $v \notin V(\De)$ and seek for contradiction. Without loss of generality, assume that $v$ is not in the subspace $H$ generated by $e_{1},\ldots, e_{k-1}$ and let $\pi_H$ be the projection from $H$. Then the image $\De_{H}:=\pi_H (\De)$ is a Klyachko polytope, $\pi_H(v) \in \De_{H}$ and $\pi_H(v) \notin V(\De_{H})\cup\{0\}$.
\\Since $W^2_d$ is terminal for $d \ge 2$, we obtain terminality by induction.
\end{proof} 

We analyse smoothness, together with $\Q$-factoriality, for Klyachko varieties. It turns out that these properties depend on some divisibility conditions on the indices $d$ and $k$ (cf. \cite{VK85}).

Let us fix some notation. For any positive $d$ and $k$ let $\overline{d}_k$ be the smallest non-negative integer $r$ which verifies $d \equiv r  \mod k$.

Fix integers $k\ge 2$ and $h\ge 1$ and let $\{x_1, \ldots, x_d\}$ be coordinates on $N_\Q$. Then define the following linear form on $N_\Q$:
$$
L_{k,h} := \sum_{i=0}^{k-3}(x_{h+ik} +x_{h+ik+1}+ \ldots +x_{h+ik+k-2}-(k-1)x_{h+ik+(k-1)}).
$$

The following proposition already appeared in \cite{VK85}, in a different notation.

\begin{proposition}\label{smoothskansen}
The Klyachko variety $W_{d}^k$ is smooth if $\gcd(d-1,k)=1$.
\\If $\gcd(d-1,k) \ne 1$, then $W_d^k$ is not $\Q$-factorial. 
\end{proposition}

\begin{proof}
The polytope $\De:=\De^k_{d}$ is not simplicial for $d=(k-1)^2$, since the hyperplane $\{L_{k,1} + x_{d}=1\}$
supports a facet of $\De$ with $k(k-1)$ vertices.
\\On the other hand, we claim that the polytope $\De$ is smooth for $d=k(k-1)$. To show this, one can see that any facet of $\De$ containing the vertex $(1,1,\ldots,1)$ also contains at least $(k-1)(k-2)+1$ elements of the standard basis. This implies that the hyperplane
$$
\{a_1x_1 + \ldots + a_dx_d=1\}
$$ 
supporting the facet has (exactly as for the hyperplane $\{L_{k,1} + x_d=1\}$): 
\begin{itemize}
\item $(k-1)(k-2)+1$ coefficients equal to 1;
\item $k-2$ coefficients equal to $-(k-1)$;
\item $k-1$ coefficients equal to 0. 
\end{itemize}
One can verify that the vertices of all these facets give a basis of $N_\Q$. Using the transitivity of $\Aut(\De)$ we obtain the claim.

The general result on smoothness is proved via induction on $k$ and $d$. Two cases are easy: 
\begin{enumerate}
\item $k=2$ and any $d$; 
\item $d=2$.
\end{enumerate}

Take $\De_d^k$ with $k,d \ge 3$: if $d <(k-1)^2$ then $\De^k_d \cong \De_d^{l+1}$, where $l:=k-\overline{d}_k$ and $d=l(k-1)$. Since $\gcd(d-1,k)=\gcd(l+1,k)=\gcd(l+1,d-1)$, we conclude by induction of $k$. 

Assume now $d > k(k-1)$. Define $h:=d-k(k-1)$ and take the plane $H$ generated by $\{e_{h+1},e_{h+2}, \ldots, e_d\}$ with projection $\pi_H$. Let define $\De_H:=\pi_H(\De^k_d)$; then $\De_H=\De^k_{h}$ and $\gcd(d-1,k)=\gcd(h-1,k)$. For any facet $\mathcal{F}$ of $\De^k_{h}$  supported on the hyperplane $\{P(x_1, \ldots,x_{h})=1\}$ we get a facet $\mathcal{F}'$ of $\De^k_d$ supported on $\{P+L_{k,(h+1)}+(x_{d-k+1} + \ldots +x_{d-1}-(k-1)x_{d})=1\}.$
\\Observe that $|V(\mathcal{F}')|= |V(\mathcal{F})| + k(k-1)$. So if $\De^k_h$ is not simplicial, neither $\De^k_d$ is so. Analogously, one checks that $\De^k_d$ is smooth if and only if $\De^k_h$ is so. We conclude via induction on $d$.
\end{proof}

\begin{remark}
As a consequence of the previous proposition, if $k$ is a prime number then $W_d^k$ is smooth, unless $d \equiv 1 \mod k$.
\end{remark}

\subsection{Low dimension}

The results and the methods of the previous subsections are enough to classify all vertex-transitive Fano manifolds up to dimension 7. 
The result is confirmed by Table \ref{t1}, which collects the Fano toric manifolds up to dimension 8 which are fibre-like\footnote{The table appeared in \cite{CFST_16} 
and has been obtained using the software MAGMA together with the Graded Ring Database \cite{GRDB} (for further details on the classification, 
cf. \cite{Obro})}.

\begin{proposition}\label{classification_low}
Let $F=F(\De)$ be a $d$-dimensional vertex-transitive Fano manifold. If $d \le 7$, then $F$ is a power of projective spaces or Klyachko manifolds.
\end{proposition}

\begin{proof}
The result can be proven using the software MAGMA together with the
classification of smooth toric Fano varieties from the Graded Ring
Database \cite{GRDB} (cf. Table \ref{t1}): giving as input a list of smooth Fano polytopes, 
MAGMA can check in which cases $\Aut(\Delta)$ acts transitively on the vertexes\footnote{We briefly describe the MAGMA code. Given an integer $i$, the function PolytopeSmoothFano($i$) gives the polytope of the $i$-th toric Fano in the Graded Ring Database; we denote by $N$ the number of Fano polytopes in the database; at the end of a run of the following code, the variable Pol will contain the list of vertex-transitive toric Fano polytopes in the Graded Ring Database.
\\Pol:=[**];
\\for $i:=1$ to $N$ do if $\#\{\{v * G$ : $G$ in AutomorphismGroup(PolytopeSmoothFano($i$))$\}$ : 
\\$v$ in Vertices(PolytopeSmoothFano($i$))$\}$
\\$-$Dimension(FixedSubspaceToPolyhedron(AutomorphismGroup(PolytopeSmoothFano($i$)))) eq 1  then  Pol:=Append(Pol,PolytopeSmoothFano($i$));
\\end if; end for;
}.
\\When the dimension is at most 4, we are able to provide the
following short argument, which does not require computer computations.

Let us start with $d=2$. If $\De$ is not 2-neighbourly, then Proposition \ref{2-neighbourly} implies that $F \cong \pp^1 \times \pp^1$ or $F \cong V_2$. 
If $\De$ is 2-neighbourly, then there is an extremal collection $P=\{x_1,x_2,x_3\}$ for which $\sigma(P)=0$ and so, by Proposition \ref{extremal}, we have $F\cong \pp^2$.

Assume now $d=3$. If $\De$ is not 2-neighbourly, then Proposition \ref{2-neighbourly} implies that $F \cong (\pp^1)^3$.
If $\De$ is 2-neighbourly, then the extremal relations could only be of the form $x_1+x_2+x_3=0$ (contradiction by Proposition \ref{extremal}), $x_1+x_2+x_3=y_1$ (contradiction by Lemma \ref{divisorial}) or $x_1+x_2+x_3+x_4=0$. In this last case, $F \cong \pp^3$ by Proposition \ref{extremal}. 

Assume finally that $d=4$. 
If $\De$ is not 2-neighbourly, then by Proposition \ref{2-neighbourly} we get $X \cong (\pp^1)^4$, $X\cong (V_2)^2$ or $X \cong V_4$.
If there is an extremal relation of the form $x_1+x_2+x_3=0$, then by Proposition \ref{extremal} we have  $X\cong \pp^2 \times \pp^2$.

Hence assume that $\De$ is 2-neighbourly and let $P$ be an extremal primitive collection. By Theorem \ref{sum} we have $|P|+|\sigma(P)| \le 5$. By Lemma \ref{divisorial} we conclude that there is an extremal relation of the form $x_1+x_2+x_3=y_1+y_2$ or $x_1+\ldots+x_5=0$. In the second case $X \cong \pp^4$ and so we can assume to have $x_1+x_2+x_3=y_1+y_2$.
From here it is not difficult to see that one should have $|V(\De)| \ge 12$, which is impossible by \cite[Theorem 1]{Cas06}.
\end{proof}

\begin{table}[b]
\begin{center}
\begin{tabular}{*{4}{c}}
Dimension &  \# Vertices & Description  & ID\\
\hline
 $2$  & $6$ & $V_2$  & $2$  \\
 2 &  4 & $\pp^1 \times \pp^1$  & 4  \\
  2 &  3 & $\pp^2$  & 5 \\
\hline
 3 & 6 & $(\pp^1)^3$  & 21   \\
 3 & 4 & $\pp^3$ &  23     \\
\hline
4 &  10 & $V_4$  & 63   \\
4 & 12 & $V_2 \times V_2$  & 100 \\
4 & 8 & $(\pp^1)^4$  & 142  \\
4 & 6 & $\pp^2 \times \pp^2$  & 146    \\
4 & 5 & $\pp^4$  & 147    \\
\hline
5 & 10 & $(\pp^1)^5$  & 1003    \\
5 & 6 & $\pp^5$  & 1013    \\
\hline
6 &  14 & $V_6$  & 1930   \\
6 & 12 & $W_6^3$  & 5817 \\
6 & 18 & $(V_2)^3$  & 7568  \\
6 & 12 & $(\pp^1)^6$  & 8611  \\
6 & 9 & $(\pp^2)^3$  & 8631    \\
6 & 8 & $(\pp^3)^2$  & 8634    \\
6 & 7 & $\pp^6$  & 8635    \\
\hline
7 & 14 & $(\pp^1)^7$  & 80835    \\
7 & 8 & $\pp^7$  & 80891   \\
\hline
8 & 18 & $V_8$  & 106303   \\
8 & 15 & $W_8^3$  & 277415  \\
8 & 20 & $(V_4)^2$  & 442179  \\
8 & 24 & $(V_2)^4$  & 790981  \\
8 & 12 & $\tilde{W}$  & 830429    \\
8 & 16 & $(\pp^1)^8$	 & 830635    \\
8 & 12 & $(\pp^2)^4$  & 830767    \\
8 & 10 & $(\pp^4)^2$  & 830782    \\
8 & 9 & $\pp^8$  & 830783    \\
\end{tabular}
\end{center}
\caption{Toric Fano manifolds of dimension at most $8$ that are fibre-like (the entry of the last column is the ID number in the Graded Ring Database \cite{GRDB}).}
\label{t1}
\end{table}

\begin{rem}\label{super_skansen}
The 8-dimensional polytope denoted by $\tilde{W}$ in Table \ref{t1} is not a Klyachko variety and we do not have a classic description of it. 
\end{rem}


\section{Fano manifolds of high index}
\label{sec_index}

An important invariant of a Fano manifold $F$ is its index, $i_F$, defined as the largest integer that divides $-K_F$ in $\Pic(F)$. 
There is a complete classification of Fano manifolds with index $i_F \ge n-2$ and $i_F \ge (n+1)/2$.
In this section we investigate the fibre-likeness of these varieties, assuming $n\ge 4$ (fibre-like Fano 3-folds have been classified in \cite{CFST_16}).

In \cite{Wis_91}, Wi\'sniewski classified Fano manifolds with index $i_F \ge (n+1)/2$. 
Let us denote by $Q^{j}\subset\pp^{j+1}$ the $j$-dimensional smooth projective quadric and by $T_{\pp^l}$ the tangent bundle of $\pp^l$.

\begin{theorem}\cite{Wis_91}\label{wis thm}
Let $F$ be a Fano manifold of dimension $n$ and index $i_F \ge (n+1)/2$. Then $F$ verifies one of the following:
\begin{enumerate}[(i)]
\item $\rho(F)=1$;
\item $n$ is even and $F\simeq \pp^{\frac{n}{2}}\times\pp^{\frac{n}{2}}$;
\item $n$ is odd and $F\simeq \pp^{\frac{n-1}{2}}\times Q^{\frac{n+1}{2}}$;
\item $n$ is odd and $F\simeq \pp\Big(T_{\pp^{\frac{n+1}{2}}}\Big)$; 
\item $n$ is odd, $\oo=\oo_{\pp^{\frac{n+1}{2}}}$ and $F\simeq \pp\Big(\oo(1)\oplus \oo^{\frac{n-1}{2}}\Big)$.
\end{enumerate}
\end{theorem}

Let us note that in case (iv) (resp. (v)) of the above theorem $F$ can be alternatively described as a smooth divisor of degree $(1,1)$ in $\pp^{\frac{n+1}{2}}\times\pp^{\frac{n+1}{2}}$
(resp. as the blow-up of $\pp^n$ along a linear $\pp^{\frac{n-3}{2}}$).

Looking at the above list, we are able to classify fibre-like Fano manifolds with high index.

\begin{proposition}\label{wis index}
Let $F$ be a fibre-like Fano manifold of dimension $n\ge 4$, index $i_F \ge (n+1)/2$ and $\rho(F)>1$. \newline
If $n$ is even then $F$ is isomorphic to $F\simeq \pp^{\frac{n}{2}}\times\pp^{\frac{n}{2}}$.\newline
If $n$ is odd then $F$ is isomorphic to  $F\simeq \pp\Big(T_{\pp^{\frac{n+1}{2}}}\Big)$.
\end{proposition}

\begin{proof}
Use Corollary \ref{contraction} to show that cases $(iii)$ and $(v)$ are not fibre-like: $(iii)$ is clear, while $(v)$ comes with a divisorial contraction to $\pp^n$ and a fibration to $\pp^{\frac{n+1}{2}}$. Fibre-likeness of $(ii)$ is a consequence of Theorem \ref{suff_criterion}, where $G=\zz/2\zz$ exchanges the two factors, and case $(iv)$ follows by Corollary \ref{complete_intersections}.
\end{proof}

In \cite{KO_73} Kobayashi and Ochai proved that $i_F \le n+1$, where $n=\dim X$ and equality holds if and only if $F \cong \mathbb P^n$. 
They also showed that $i_F = n$ if and only if $F$ is a quadric hypersurface.
Fano manifolds with index $n-1$ are called \emph{del Pezzo manifolds} and they have been classified by Fujita \cite{F_82a} and \cite{F_82b}, 
while Fano manifolds with index $n-2$ are called \emph{Mukai manifolds} and their classification appeared in \cite{M_89}. 
In dimension $n\ge 3$, del Pezzo manifolds have index $i_F \ge (n+1)/2$ and they have already been studied in Proposition \ref{wis index}. 
For $n\ge 5$, also Mukai manifolds are included in Wi\'sniewski's list, so we only need to study the case $n=4$, $i_F=2$ (see \cite[Table 12.7 ]{AG_99} for the complete list).

\begin{proposition}\label{mukai}
Let $F$ be a $4$-dimensional fibre-like Fano manifold of index $i_F=2$ and $\rho(F)>1$. Then $F$ is isomorphic to one of the following:
\begin{enumerate}[(i)]
\item a double cover of $\pp^2 \times \pp^2$ branched along a degree--$(2,2)$ divisor;
\item an intersection of two degree--$(1,1)$ divisors in $\pp^3\times\pp^3$;
\item $\pp^1\times\pp^1\times\pp^1\times\pp^1$.
\end{enumerate}
\end{proposition}

\begin{proof}
We follow the enumeration in \cite[Table 12.7 ]{AG_99} of the 18 families. Applying Corollary \ref{contraction} to the cases 
$$(1), (2), (3), (5), (6), (8), (9), (14), (16) \mbox{ and } (17)$$
we immediately see that they are not fibre-like. Cases
$$(10) \mbox{ and } (12)$$
are also not fibre-like, since they are obtained as a blow-up of $Q^4$ but come with a $2$-dimensional fibration over $\pp^{n-2}$. Cases
$$(13) \mbox{ and } (15)$$
are $\pp^1$-bundles over $\pp^3$ or $Q^3$ and the other ray of the nef cone corresponds to the contraction of the section. The case which requires more care is 
$(11)$, in which case $F$ is isomorphic to the projectivisation of the null-correlation bundle over $\pp^3$. 
Although the two extremal rays of the nef cone of $F$ both yield fibrations, 
the image of the fibration associated to the ray not inducing the bundle structure is the quadric $Q^3$, see \cite[Proposition 3.4]{SzW_90}. 
Hence, $F$ is not fibre-like. \newline
We now prove fibre-likeness for the remaining cases. Case $(18)$ (corresponding to case $(iii)$ in our list), 
has an action of $S_4$ and is clearly fibre-like because of Theorem \ref{suff_criterion}. 
For Case $(7)$ (corresponding to case $(ii)$), we can directly apply Corollary \ref{complete_intersections}. 
Let us analyse now the case $(4)$ (corresponding to case $(i)$ in our list): it is obtained as a member of the linear system $|2H_1+2H_2|$ in the toric variety $Z$ with weight data
\begin{center}
\vspace{.2cm}
{
\begin{tabular}{*{7}{c} | *{1}{c} }\label{toric_data}
$x_0$ & $x_1$ & $x_2$  & $y_0$ & $y_1$ & $y_2$ & $z$ & \\
\hline
1 & 1 & 1 & 0 & 0 & 0 & 1 & $H_1$\\
0 & 0 & 0 & 1 & 1 & 1 & 1 & $H_2$\\
\end{tabular}.
}
\end{center}
Since $Z$ comes with a $\zz/2\zz$-action exchanging the divisors $H_1$ and $H_2$ we have $\dim \NS(Z)_\qq^{\zz/2\zz}=1$ 
and we can apply \cite[Theorem 4.5]{CFST_16} to conclude that case $(i)$ is fibre-like. 
\end{proof}


\section{Fano manifolds with high Picard number}
\label{sec_fourfolds}

Fano manifolds of a given dimension form a bounded family, so their Picard number is bounded. 
In spite of this boundeness result, their classification in dimension at least four is an open and rather difficult problem. 
A first step towards such classification would be to identify an effective bound on the Picard number of those Fano manifolds 
that are not a product of lower dimensional manifolds. Already this simpler problem is actually quite difficult. 
So far, the only known examples of families of Fano manifolds of dimension $n\geq 4$,  which are not product and have Picard number at least $n+4$ are
\begin{enumerate}[(i)]
\item a birational model of the blow up of $\mathbb{P}^n$ in $n+3$ points in general position, with $n\ge4$ and even;
\item a birational model of the blow up of $\mathbb{P}^4$ in $8$ points in general position.
\end{enumerate}
The first family appears in any even dimension, the second example is sporadic. 
In this section we are going to show that all these examples are fibre-like. 
The first family with $n$ odd gives non-$\qq$-factorial Fano varieties of Picard rank $1$, 
see \cite[page 3029]{AC17}. 
We do not know if there exists a general connection between having high Picard number and being fibre-like. 

The sporadic example is discussed in \cite{CCF17}. 
There the authors prove, following \cite{M05}, that the Fano manifold under investigation is 
isomorphic to the moduli space of rank two vector bundles on a del Pezzo surface of degree one. 
Varying the stability conditions, the authors can explicitly describe the birational geometry of the Fano manifolds. 
Thanks to this analysis, it is possible to describe the automorphism group of the manifold, and to show the following result. 

\begin{theorem}\cite[Proposition 6.22]{CCF17}\label{thm_CCF}
The Fano model of $\pp^4$ blown-up in 8 points in general position is fibre-like.
\end{theorem}

We now focus on the first example. We use the results of \cite{AC17} and \cite[Section 3]{R72}. Take an even integer $n=2m\ge 2$ and consider a smooth complete intersection $Z$ of two quadrics in $\pp^{n+2}$. 
Let $\mathcal{F}_{m-1}=\mathcal{F}_{m-1}(Z)$ be the variety of $(m-1)$-planes in $Z$. This is a smooth Fano variety of dimension $n$. It can be seen as a higher dimensional generalisation 
of the quartic del Pezzo surface and has been extensively studied in the recent work \cite{AC17}. 

The geometry of $\mathcal{F}_{m-1}$ can be studied from another point of view, which we briefly recall here (see the survey \cite{C17} for the notation about Mori Dream Spaces). Let $X:=X^n_r$ be the blow up of $\pp^n$ at $r$ points in general position. 
Then it follows from \cite{M05} and \cite[Theorem 1.3]{CT06} that $X$ is a Mori Dream Space if and only if $n$ and $r$ verify the inequality
\begin{equation}\label{MDS_inequality}
\frac{1}{n+1}+\frac{1}{r-n-1}> \frac{1}{2},
\end{equation}
cf. \cite[Example 3.6]{C17}.

The manifolds appearing in Theorem \ref{thm_CCF} are, in this notation, obtained as Fano models of~$X^4_8$.

Look at $X^n_{n+3}$, with $n\ge 2$ even. For this class, inequality \eqref{MDS_inequality} holds, so $X^n_{n+3}$ is a Mori Dream Space and we can consider its Fano model $F^n_{n+3}$. 
Bauer in \cite{B91} proved that $X^n_{n+3}$ and $\mathcal{F}_{m-1}$ are isomorphic in codimension one -- see also \cite[Theorem 1.4]{AC17}. 
By \cite[Remark 4.10]{AC17} it follows that $\mathcal{F}_{m-1}$ is actually isomorphic to $F^n_{n+3}$. In \cite[Proposition 7.1]{AC17} the authors
describe the automorphism groups of $\mathcal{F}_{m-1}$ showing that
\[
\bigg(\big(\zz\big\slash 2\zz\big)^{n+2} \subseteq \bigg) \Aut(\mathcal{F}_{m-1})\subseteq W(D_{n+3}) \bigg(= \big(\zz\big\slash 2\zz\big)^{n+2} \rtimes S_{n+2}\bigg),
\]
where $W(D_{n+3})$ is the Weyl group of automorphism of a $D_{n+3}$-lattice. 
The inclusion $(\mathbb{Z}\slash 2\mathbb{Z})^{n+2} \subset \Aut(\mathcal{F}_{m-1})$ is an actual equality for a general choice of $\mathcal{F}_{m-1}$. 
The action of $(\zz\big\slash 2\zz\big)^{n+2} $ can be described by presenting $Z$ as the locus
$$
\sum_{i=0}^{n+3}x_i^2=\sum_{i=0}^{n+3}\lambda_i x_i^2=0.
$$
Then the group acts by changing the signs of the coordinates. We can use this to prove the following.

\begin{theorem}\label{main_AC}
Let $n=2m\ge 4$ be an integer. Then the smooth $n$-dimensional Fano variety $\mathcal{F}_{m-1}(Z)$ of $(m-1)$-planes in the intersection of two quadrics $Z\subset \pp^{n+2}$ is fibre-like.
\end{theorem}

\begin{proof}
Consider the isomorphism $H^2(\mathcal{F}_{m-1}, \zz) \simeq \NS(\mathcal{F}_{m-1}).$ The action of $G:=\big(\zz\big\slash 2\zz\big)^{n+2}$ via pseudo-automorphisms of $X^n_{n+3}$ is explicitly described in \cite[Sections 4.4-4.6]{D04} 
(see also \cite[Remark 7.2]{AC17}). Let $x_0, \ldots, x_{n+2} \in \pp^n$ be blown-up points; we can assume that the first $n+1$ are the coordinate points and
\begin{itemize}
\item $x_{n+1}=[1:\ldots : 1]$;
\item $x_{n+2}=[c_0:\ldots : c_{n}]$.
\end{itemize}
The pseudo-automorphism $\phi_{n+2,n+3}\colon X^n_{n+3} \to X^n_{n+3}$ is defined on $\pp^n$ as $\rho \circ \iota$, 
where $\iota\colon \pp^n \dashrightarrow \pp^n$ is the standard Cremona involution and $\rho$ is the 
diagonal projective transformation $[t_0:\ldots : t_n] \mapsto [c_0t_0:\ldots : c_nt_n]$. 
Analogously, one defines $\phi_{i,j}$ with $i<j$, which exchanges the exceptional divisors of $X^n_{n+3}$ and 
fixes $K_{X^n_{n+3}}$. This implies that $\NS(X^n_{n+3})_{\qq}^G=\qq K_{X^n_{n+3}} $. 
Moreover, the analysis in \cite[Proposition 5.4]{AC17} implies that $\NS(X^n_{n+3})^{G}\simeq \NS(\mathcal{F}_{m-1})^{G}$. 
So we apply Theorem \ref{suff_criterion} to conclude.
\end{proof}


\section{Open questions}
\label{sec_questions}

We conclude this note with some questions regarding fibre-like varieties that are still open.

Since all the examples of smooth toric fibre-like Fano varieties are vertex transitive, we ask the following:

\begin{question}\label{v_tran}
	Is any fibre-like toric Fano manifold vertex-transitive?
\end{question}

Answering this question affermatively would give a complete classification of fibre-like toric Fano manifolds.

On a different note, the study of fibre-like Fano varieties in positive characteristic 
seems to be still very far from being satisfactory.\newline
After the recent developments for the MMP in positive characteristic for threefods 
(cf. \cite{Hacon_Xu_2015}, \cite{Birkar_2016}, \cite{BW_2017}), the picture that we delineated 
in the Introduction holds almost in the same way in characteristic $>5$ as long as we only
focus in dimension $2$ and $3$. Hence, it is natural to try to extend the results in \cite{CFST_16} 
to positive characteristic. 
If $k$ is any algebraically closed field and $F$ is a smooth Fano variety over $k$, the definition 
of fibre-likeness still makes sense and one can in particular ask the following:

\begin{question}\label{char_p}
 If $\ch k=p>0$, are there sufficient or necessary conditions that determine whether a smooth Fano $F$ is fibre-like?
\end{question}

At present time, the situation appears to be quite obscure:
we do not even know if a del Pezzo surface of degree $8$ 
is fibre-like in positive characteristic. The approach outlined in \cite{CFST_16} relaying on the study of a suitable 
monodromy action on the N\'eron-Severi does not generalize directly to this case. 

\bibliographystyle{alpha}
\bibliography{./bibliografia}

\end{document}